\documentclass[a4paper,10pt]{amsart}
\usepackage[arrow,matrix]{xy}
\usepackage{amsmath,amssymb,amscd,bbm,amsthm,mathrsfs,dsfont,enumerate}
\theoremstyle{plain} \textwidth=36pc \textheight=51pc

\newtheorem{theorem}{Theorem}[section]
\newtheorem{lemma}[theorem]{Lemma}
\newtheorem{example}[theorem]{Example}
\newtheorem{proposition}[theorem]{Proposition}
\newtheorem{corollary}[theorem]{Corollary}
\theoremstyle{definition}
\newtheorem{definition}[theorem]{Definition}
\newtheorem{remark}[theorem]{Remark}
\numberwithin{equation}{section}

\DeclareMathOperator{\End}{End} \DeclareMathOperator{\Ext}{Ext}

 \DeclareMathOperator{\h}{H}

\DeclareMathOperator{\im}{Im}
\DeclareMathOperator{\e}{\epsilon}
\DeclareMathOperator{\GK }{GK dim}

\DeclareMathOperator{\gr }{gr}
\DeclareMathOperator{\ad }{ad}

\begin{document}
\title[Pointed Hopf algebras of finite GK-dimension]
 {Properties of pointed and connected Hopf algebras of finite Gelfand-Kirillov dimension} 
\maketitle

\begin{abstract}
Let $H$ be a pointed Hopf algebra. We show that under some mild assumptions $H$ and its associated graded Hopf algebra $\gr H$ have the same Gelfand-Kirillov dimension. As an application, we prove that the Gelfand-Kirillov dimension of a connected Hopf algebra is either infinity or a positive integer. We also classify connected Hopf algebras of GK-dimension three over an algebraically closed field of characteristic zero.
\end{abstract}


\section{Introductioin}

The Gelfand-Kirillov dimension (or GK-dimension for short) has been a useful tool for investigating infinite-dimensional Hopf algebras. For example, Hopf algebras of low GK-dimensions are studied in \cite{BZ,GZ,Li,Z,WZZ1,WZZ2}. 

It is well known that every Hopf algebra $H$ has a coradical filtration $\{H_n\}_{n=0}^{\infty}$. If $H$ is pointed with group-like elements $G$, then the associated graded algebra of $H$ with respect to the filtration $\{H_n\}_{n=0}^{\infty}$ is a graded Hopf algebra, which we denote by $\gr H$. The structure of $\gr H$ is relatively easier in the sense that it has a nice decomposition $\gr H\cong R\#kG$, where $R$ is a certain graded subalgebra of $\gr H$ (see \cite[Theorem 3]{R}). In the first part of this paper, we clarify the behavior of the GK-dimension of a pointed Hopf algebra under taking associated graded algebra. In fact, we prove the following Theorem.
\begin{theorem}[(Theorem $\ref{equality}$)]\label{First}
Retain the above notation. If $R$ is finitely generated, then
\[\GK R+ \GK kG=\GK \gr H=\GK H.\]
\end{theorem}
The first equality follows from Lemma \ref{addition}, which is a generalized version of \cite[Lemma 5.5]{Z}. The proof of the second equality depends heavily on Takeuchi's construction of free Hopf algebras, which we will review briefly in Section \ref{Takeuchi}.

An interesting phenomenon is that the GK-dimension of every known Hopf algebra is either infinity or a non-negative integer. So it is tempting to  conjecture that this is always true for any Hopf algebra. As positive evidence for this conjecture, we prove in Theorem \ref{integer} that the GK-dimension of a connected Hopf algebra over an algebraically closed field of characteristic zero is either infinity or a positive integer. This is basically a consequence of Theorem \ref{equality} and the following result.

\begin{theorem}[(Proposition $\ref{polynomial}$)]
Let $K$ be a connected coradically graded Hopf algebra and assume that the base field $k$ is algebraically closed of characteristic $0$. If $K$ is finitely generated, then $K$ is isomorphic to the polynomial ring in $\ell$ variables for some $\ell\ge0$ as algebras.
\end{theorem}

For the definition of coradically graded Hopf algebras, one can refer to Definition \ref{gradedHopf}. Notice that if $H$ is a connected Hopf algebra, then $\gr H$ is a connected coradically graded Hopf algebra (see Remark \ref{coradically}).

In the last two sections, with the help of the results from previous sections, we classify connected Hopf algebras of GK-dimension three over an algebraically closed field of characteristic zero. The result can be stated as follows. Note that we use $P(H)$ to denote the space of primitive elements of $H$.

\begin{theorem}[(Theorem $\ref{classification}$)]
Let $H$ be a connected Hopf algebra of GK-dimension three (over an algebraically closed field of characteristic $0$). Then $H$ is isomorphic to one of the following:
\begin{enumerate}
\item[\textup{(I)}] The enveloping algebra $U(\mathfrak{g})$ for some three-dimensional Lie algebra $\mathfrak{g}$;
\item[\textup{(II)}] The Hopf algebras $A(0, 0, 0)$, $A(0, 0, 1)$, $A(1, 1, 1)$ or $A(1, \lambda, 0)$ from Example \ref{typeA} for some $\lambda\in k$;

\item[\textup{(III)}] The Hopf algebras $B(\lambda)$ from Example \ref{typeB} for some $\lambda\in k$.
\end{enumerate}
\end{theorem}
\noindent{\bf Acknowledgment.} The author thanks Professor James Zhang, Professor Dingguo Wang, Xingting Wang and Cris Negron for useful conversations and for their careful reading of this paper. The research was partially supported by the US National Science Foundation.

\section{Preliminaries}\label{pre}
Throughout this paper, $k$ denotes a base field. All algebras, coalgebras and tensor products are taken over $k$ unless otherwise stated. Given a group $G$, we will use $kG$ to denote the group algebra of $G$ over $k$.

For a coalgebra $C$, we denote the comultiplication and the counit by $\Delta$ and $\e$, respectively. Let $G(C)$ be the set of the group-like elements in $C$, and $C^+$ be the kernel of the counit. The \textbf{coradical}  $C_0$ of $C$ is defined to be the sum of all simple subcoalgebras of $C$. The coalgebra $C$ is called \textbf{pointed} if $C_0=kG(C)$, and \textbf{connected} if $C_0$ is one-dimensional. Also, we use $\{C_n\}_{n=0}^{\infty}$ to denote the coradical filtration of $C$ \cite[5.2.1]{Mont}. The coalgebra $C$ is called \textbf{coradically finite} if $\dim_k C_n<\infty$ for any $n$. For a Hopf algebra $H$, we use $P(H)$ to denote the space of primitive elements of $H$.

Let $C=\bigoplus\limits_{i=0}^{\infty }C(i)$ be a graded coalgebra. We say that $C$ is \textbf{coradically graded} if $C_0=C(0)$ and $C_1=C(0)\bigoplus C(1)$. If $C$ is coradically graded, then as shown by  \cite[Lemma]{CM}, $C_n=\bigoplus_{i\le n}C(i)$ for any $n\ge 0$.  Now we recall the definition of graded Hopf algebras, which will be used intensively in this paper.

\begin{definition}\label{gradedHopf}
Let $H$ be a Hopf algebra with antipode $S$. If 
\begin{enumerate}
\item[(1)] $H=\bigoplus\limits_{i=0}^{\infty }H(i)$ is a graded algebra,
\item[(2)] $H=\bigoplus\limits_{i=0}^{\infty }H(i)$ is a graded coalgebra,
\item [(3)]$S(H(n))\subset H(n)$ for any $n\ge 0$,
\end{enumerate}
then $H=\bigoplus\limits_{i=0}^{\infty }H(i)$ is called a \textbf{graded Hopf algebra}. If in addition, 
\begin{enumerate}
\item[(4)]$H=\bigoplus\limits_{i=0}^{\infty }H(i)$ is a coradically graded coalgebra,
\end{enumerate}
then $H$ is called a \textbf{coradically graded Hopf algebra}.
\end{definition}

\begin{remark}\label{coradically}
It turns out that the notion of coradically graded Hopf algebra is very natural. For example, let $H$ be a pointed Hopf algebra with coradical filtration $\{H_n\}_{n\ge 0}$. Then the associated graded space $\gr H=\bigoplus_{n\ge 0}H_n/H_{n-1}$ is a graded Hopf algebra \cite[p. 62]{Mont}. Moreover, as mentioned in \cite[Definition 1.13]{AS}, $\gr H$ is a coradically graded coalgebra. Therefore, $\gr H$ is a coradically graded Hopf algebra.
\end{remark}

\section{Takeuchi's construction of free Hopf alegbras}\label{Takeuchi}
In \cite{T}, Takeuchi proved that for any coalgebra $C$ there exists a Hopf algebra $\mathcal{H}(C)$ characterized by the following universal property:
\begin{itemize}
\item[(1)] There is a coalgebra map $i: C\rightarrow \mathcal{H}(C)$,
\item[(2)] For any Hopf algebra $H$ and coalgebra map $f: C\rightarrow H$, there is a Hopf algebra map $f': \mathcal{H}(C)\rightarrow H$ such that $f=f'i$.
\end{itemize}
The Hopf algebra $\mathcal{H}(C)$ is called the free Hopf algebra generated by $C$ \cite[Definition 1]{T}. Takeuchi showed the existence of $\mathcal{H}(C)$ by an explicit construction, which we will describe briefly.

Let $V=\bigoplus_{i=0}^{\infty}V_i$ where $V_i=C$ if $i$ is even and $V_i=C^{cop}$ if $i$ is odd. Notice that $V$ has a natural coalgebra structure. Let $S: V\rightarrow V^{cop}$ be the coalgebra map sending $(x_0, x_1, x_2, \cdots)$ to $(0, x_0, x_1, x_2, \cdots)$. Then $S$ induces a bialgebra map $S: T(V)\rightarrow T(V)^{op, cop}$. Let $I$ be the two-sided ideal of $T(V)$ generated by the set
$$\{S*Id(x)-\e(x)1\, |\, x\in V\}\bigcup \{Id*S(x)-\e(x)1 \,|\, x\in V\},$$
where $*$ represents the convolution. Moreover, $I$ is a coideal and $S(I)\subset I$. By \cite[Lemma 1]{T}, the Hopf algebra $T(V)/I$, with the antipode induced from the map $S$, is the universal object $\mathcal{H}(C)$.

Takeuchi's construction generalizes the notion of a free group generated by a set in the following sense.
\begin{proposition}[{\cite[Lemma 34]{T}}]
$\mathcal{H}(kG(C))=k\langle G(C)\rangle$, where $\langle G(C)\rangle$ is the free group generated by the set $G(C)$.
\end{proposition}

Let $C$ be a coalgebra with coradical $C_0$ and let $C=C_0\oplus V$ be a decomposition of $C$ as a $k$-space. Then $\mathcal{H}(C)$ can be realized by giving a natural Hopf structure on the algebra $\mathcal{H} (C_0)\amalg T(V)$ \cite[\S 6]{T}, where $\amalg $ denotes the coproduct in the category of algebras. Now the canonical coalgebra map $i: C\rightarrow \mathcal{H}(C)$ can be identified with the map induced by maps $C_0\rightarrow \mathcal{H}(C_0)$ and $V\rightarrow T(V)$. By this characterization, we have
\begin{lemma}[{\cite[Theorem 35]{T}}]\label{basis}
Suppose that $C$ is a pointed coalgebra and $G(C)\cup B$ is a $k$-basis for $C$. Let $X=G(C)\cup G(C)^{-1}\cup B$ and let $Y$ be the set of finite sequences $(x_1, \cdots, x_n)$ of elements of $X$ such that $(x_i, x_{i+1})$ is not of the form $(g^{\pm1}, g^{\mp1})$ where $g\in G(C)$. Set
$$\overline{x}=x_1\cdots x_n\in \mathcal{H}(C) \,\,\text{for}\,\, x=(x_1, \cdots, x_n)\in Y,$$
where by abuse of notation we still use $x_j$ for its image in $\mathcal{H}(C)$ under the canonical map $i: C\rightarrow \mathcal{H}(C)$.
Then $\{\overline{x} \,|\, x\in Y\}$ forms a $k$-basis for $\mathcal{H}(C)$.
\end{lemma}
Now we are able to prove the following proposition, which determines the coradical of $\mathcal{H}(C)$ when $C$ is pointed. For a given $k$-subspace $W$ of an algebra $R$ and any $n\ge 1$, let $W^n$ denote the $k$-subspace of $R$ spanned by products of $\le n$ elements in $W$. 
 
\begin{proposition}\label{coradicaluniversal}
Let $C$ be a pointed coalgebra. For any $n\ge 1$. Then the following statements are true.
\begin{enumerate} 
\item[\textup{(I)}] the Hopf algebra $\mathcal{H}(C)$ is pointed and the coradical of $\mathcal{H}(C)$ is equal to the subalgebra generated by $G(C)$ and their inverses, which is isomorphic to $k\langle G(C)\rangle$.
\item[\textup{(II)}] for any $n\ge 1$, $C^n$ is a subcoalgebra of $\mathcal{H}(C)$ and any element in $G(C^n)$ can be expressed as a product of $\le n$ elements in $G(C)$.
\end{enumerate}
\end{proposition}
\begin{proof} Denote $\mathcal{H}(C)$ by $H$. It is clear that the subalgebra of $H$ generated by $G(C)$ and their inverses is contained in $H_0$. By Lemma \ref{basis}, this subalgebra is isomorphic to $k\langle G(C)\rangle$. 

Choose a subset $B$ of $C$ such that $B=\bigcup_{i=1}^{\infty}B_i$ and $G(C)\bigcup B_1\bigcup\cdots\bigcup B_n$ is a $k$-basis of $C_n$ for any $n$. Let $V$ be the $k$-space spanned by $B$. By  \cite[\S 6]{T}, $H\cong k\langle G(C)\rangle\amalg T(V)$ as an algebra. Hence we can define a grading on $H$ by setting $\deg g=0$ for any $g\in G(C)$ and $\deg b_i=i$ for any $b_i\in B_i$. Under this grading, $H$ becomes a graded algebra (but not necessarily a graded coalgebra) and $C$ is a graded subspace of $H$. Let $H(n)$ be the $k$-subspace of $H$ spanned by homogeneous elements of degree $n$.

Now, fix a basis of $H$ described in Lemma \ref{basis} with the chosen set $B$. Then this is a basis consisting of homogeneous elements with respect to the grading defined above. For any $m\ge 0$, let $A_m=\sum_{i=0}^{m}H(i)$.
It is clear that $\{A_m\}_{m\ge 0}$ is an algebra filtration on $H$ and $A_0$ is exactly the subalgebra of $H$ generated by $G(C)$ and their inverses. Moreover, by the choice of $B$, we have $C_m\subset A_m$ for any $m\ge 0$. Let $\overline{x}=x_1\cdots x_n$ be as in Lemma \ref{basis} and suppose that $x_j\in B_{i_j}$ (here we set $B_0=G(C)\bigcup G(C)^{-1}$). Then by definition $A_m$ is spanned by $\overline{x}$ such that $\sum_j i_j\le m$. Pick such an $\overline{x}\in A_m$.
Notice that
$\Delta(\overline{x})=\Delta(x_1)\cdots\Delta(x_n)$ and for each $j$, $\Delta(x_j)\subset \sum_{t} C_t\otimes C_{i_j-t} \subset \sum_{t} A_t\otimes A_{i_j-t}$. As a consequence, 
\[\Delta(\overline{x})\in\sum_{t} A_t\otimes A_{m-t}.\]
This shows that $\{A_m\}_{m\ge 0}$ is also a coalgebra filtration. By \cite[Lemma 5.3.4]{Mont}, $H_0\subset A_0$. This proves the first statement.

For the second statement, it is easy to check that $C^n$ is a subcoalgebra of $H$. As mentioned before, $C$ is a graded $k$-subspace of $H$. This means that $C=\bigoplus_{i=0}^{\infty}C(i)$ where $C(i)=C\bigcap H(i)$. Now $G(C^n)=C^n\bigcap G(H)\subset C^n\bigcap H(0)$. Since $C$ is a graded $k$-subspace of $H$, $C^n\bigcap H(0)=C(0)^n$. Notice that $C(0)$ is spanned by $G(C)$. Hence every element in $G(C^n)$ can be expressed as a linear combination with each summand a scalar multiple of a product of $\le n$ elements in $G(C)$. Since distinct group-like elements are $k$-linearly independent by \cite[3.2.1]{Sw}, every element in $G(C^n)$ is actually a product of $\le n$ elements in $G(C)$.
\end{proof}

\begin{remark}\label{grading} 
Suppose that $C=\bigoplus_{i=0}^{\infty}C(i)$ is a pointed graded coalgebra such that $C(0)=C_0$. Now $C$ has a canonical decomposition $C=C_0\oplus V$, where $V=\bigoplus_{i=1}^{\infty}C(i)$. Then by the construction of the comultiplication and the antipode on $\mathcal{H}(C_0)\amalg  T(V)$ in \cite[Lemma 26, Lemma 27]{T}, $\mathcal{H}(C)=\mathcal{H}(C_0)\amalg  T(V)$ becomes a graded Hopf algebra, where elements in $\mathcal{H}(C_0)$ have degree $0$ and the grading on $T(V)$ is inherited from $V$. It is clear that the canonical inclusion $i: C\rightarrow \mathcal{H}(C)$ becomes a graded coalgebra map. Moreover, if there is a graded coalgebra map $f$ from $C$ to a graded Hopf algebra $H$, then the lifting Hopf algebra map $f': \mathcal{H}(C)\rightarrow H$ is also graded.
\end{remark}

\begin{corollary}\label{approx}
Let $H$ be a pointed Hopf algebra. Then every finite-dimensional subspace $V$ of $H$ is contained in a finitely generated Hopf subalgebra. 
Moreover, if $D$ is a finite-dimensional subcoalgebra of $H$ with group-like elements $G(D)$, then elements in $G(D^n)$ can be expressed as products of $\le n$ elements in $G(D)$.
\end{corollary}

\begin{proof}
Since $V$ is contained in a finite-dimensional subcoalgebra of $H$, we can assume $V=D$ is a finite-dimensional subcoalgebra. Let $C$ be a copy of $D$ as coalgebras. Then there is an injective coalgebra map $f: C\rightarrow H$ whose image is $D$. By the universal property of $\mathcal{H}(C)$, there is a Hopf algebra map $f':\mathcal{H}(C)\rightarrow H$ such that $f=f'i$, where $i$ is the inclusion $C\rightarrow \mathcal{H}(C)$. (Notice here if we do not introduce a copy $C$ of $D$, then by writing $D$ we could mean either a subcoalgebra of $H$ or a subcoalgebra of $\mathcal{H}(D)$, which may cause confusion in the proof). By Lemma \ref{basis}, $\mathcal{H}(C)$ is a finitely generated algebra. By construction, $D$ is contained in $f'(\mathcal{H}(C))$. This proves the first claim. 

Let $S=\{g_1, \cdots, g_\ell\}$ be the set of group-like elements of $C$. Then $\{f(g_1), \cdots, f(g_\ell)\}$ is the set of group-like elements of $D$. Notice that $C^n$ is a subcoalgebra of $\mathcal{H}(C)$. By \textup{(II)} of Proposition \ref{coradicaluniversal}, every group-like element of $C^n$ can be expressed as a product of $\le n$ elements in $S$. Since $C^n$ is mapped onto $D^n$ by $f'$, we have $G(D^n)=f'(G(C^n))$ by \cite[Corollary 5.3.5]{Mont}. The result then follows. 
\end{proof}

\begin{corollary}\label{0finite}
Let $H$ be a finitely generated pointed Hopf algebra. Then $H_0$ is a finitely generated algebra. In fact, if $D$ is a finite-dimensional subcoalgebra of $H$ that generates $H$ as an algebra, then $H_0$ is generated by $G(D)$.
\end{corollary}
\begin{proof}
Since every finite-dimensional subspace of $H$ is contained in a finite-dimensional subcoalgebra by \cite[5.1.1]{Mont}, we can assume that $H$ is generated as an algebra by a finite-dimensional subcoalgebra $D$. By assumption $H_0$ is spanned by $G(H)$. For any $g\in G(H)$, there exists some $n$ such that $g\in D^n$. Notice that $D^n$ is a subcoalgebra of $H$. Hence $g\in G(D^n)$. Now the result follows from the second statement of Corollary \ref{approx}.
\end{proof}

\begin{remark}\label{gradedaffine}
Let $H=\bigoplus_{i=0}^{\infty}H(i)$ be a graded pointed Hopf algebra such that $H(0)$ is spanned by all group-like elements of $H$. In this case, every finite-dimensional subspace $V$ of $H$ is contained in a finitely generated graded Hopf subalgebra of $H$. In fact, without loss of generality, we can assume that $V$ is a finite-dimensional graded subspace of $H$. By a similar argument as in \cite[Theorem 5.1.1]{Mont}, $V$ is contained in a finite-dimensional graded subcoalgebra $C$ of $H$. Then the result follows from Remark \ref{grading} and an argument similar to that of Corollary \ref{approx}.
\end{remark}

We conclude this section with a proposition regarding the GK-dimension of a pointed Hopf algebra, which is a direct consequence of Corollary \ref{approx}.
\begin{proposition}\label{localfinite}
Let $H$ be a pointed Hopf algebra. Then 
$$\GK H=\sup_{E}\GK E,$$
where the supreme is taken over all finitely generated Hopf subalgebras of $H$.
\end{proposition}

\section{Pointed Hopf algebras and their associated graded Hopf algebras}\label{associated}
Throughout this section, let $H$ be a pointed Hopf algebra with group-like elements $G$. We use $\gr H$ to denote the associated graded Hopf algebra of $H$ with respect to the coradical filtration. 
There is a canonical Hopf projection $\psi: \gr H\rightarrow H_0$. Let $R=(\gr H)^{co\psi}$, the algebra of coinvariants of $\psi$ \cite[1.5]{AS}. By definition $R$ is a graded subalgebra of $\gr H$. In fact, it is well known that $R$ is a graded braided Hopf algebra in $^{G}_G\mathcal{YD}$, the Yetter-Drinfeld category over $G$, and 
\begin{equation}\label{bidecomp}
\gr H\cong R\#H_0,
\end{equation}\label{smash}
as Hopf algebras.
Let $H_0^+$ be the $k$-space spanned by the elements of the form $1-g$ where $g\in G$. Notice that $HH_0^+$ is a coideal of $H$. Denote the coalgebra $H/HH_0^+$ by $\theta(H)$ and the coalgebra projection $H\rightarrow \theta(H)$ by $\pi_H$.
\begin{lemma}\label{flatness}
Retain the above notation. Then $H/H_n$ is a free (right) $H_0$-module for any $n\ge 0$.
\end{lemma}
\begin{proof}
By \cite[Lemma 1]{Ra}, for any $m\ge0$, $H_{m+1}/H_m$ is a free $H_0$-module. Hence inductively we see that $H_{n+i}/H_n$ is a free $H_0$-module for any $i\ge1$ and $H_{n+i}/H_n\cong \bigoplus\limits_{j\ge 0}^{i-1}H_{n+j+1}/H_{n+j}$ as $H_0$-modules. As a consequence, $H/H_n\cong \bigoplus\limits_{j\ge 0}H_{n+j+1}/H_{n+j}$ as $H_0$-modules. The result then follows.
\end{proof}

\begin{lemma}\label{intersection}
Retain the above notation and let $I=HH_0^+$. Then $I\cap H_n=H_nH_0^+$ for any $n\ge 0$.
\end{lemma}
\begin{proof}
For any $n\ge 0$, we have the short exact sequence
\[0\rightarrow H_n\rightarrow H\rightarrow H/H_n\rightarrow 0.\]
Since $H/H_n$ is a free $H_0$-module, the following sequence is exact, 
\[0\rightarrow H_n\otimes_{H_0}k\rightarrow H\otimes_{H_0}k\rightarrow (H/H_n)\otimes_{H_0}k\rightarrow 0.\]
This shows that the canonical map $H_n/H_nH_0^+\rightarrow H/I$ is injective, which implies that $I\cap H_n=H_nH_0^+$.
\end{proof}

Let $F_n$ be $\pi_H(H_n)\subset \theta(H)$. Then $\{F_n\}_{n\ge0}$ becomes a coalgebra filtration on $\theta(H)$. The following lemma is clear.
\begin{lemma}\label{inducedfil}
Suppose that $f: C\rightarrow D$ is a surjective coalgebra map and $C$ has a coalgebra filtration $\{A_n\}_{n\ge0}$. Let $B_n=f(A_n)$. Then $\{B_n\}_{n\ge 0}$ is a coalgebra filtration on $D$. Moreover, $f$ induces a surjective graded coalgebra map $\gr_AC\rightarrow \gr_B D$.
\end{lemma}
By this lemma, we see that there is a surjective graded coalgebra map $\gr H\rightarrow \gr_F\theta(H)$ induced by $\pi_H$.

\begin{proposition}\label{Filtration}
Retain the above notation. Then $\gr_F\theta(H)$ is isomorphic to $\theta(\gr H)$ as graded coalgebras.
\end{proposition}
\begin{proof}
By definition, $\theta(\gr H)=\gr H/(\gr H)H_0^+$. So we only have to show that the kernel of the map $\gr H\rightarrow \gr_F\theta(H)$ induced by $\pi_H$ is $(\gr H)H_0^+$. It suffices to prove that for any $n\ge 0$, the canonical map 
$H_{n+1}/H_n\rightarrow \pi_H(H_{n+1})/\pi_H(H_n)$ has kernel $(H_{n+1}/H_n)H_0^+$. 
Let $I=HH_0^+$. It is easy to check that the 
\[\ker (H_{n+1}/H_n\rightarrow \pi_H(H_{n+1})/\pi_H(H_n))=\frac{H_{n+1}\cap(H_n+I)}{H_n}=\frac{H_{n+1}\cap I+H_n}{H_n}.\]
By Lemma \ref{intersection}, $H_{n+1}\cap I=H_{n+1}H_0^+$. Therefore, 
\[\frac{H_{n+1}\cap I+H_n}{H_n}=\frac{H_{n+1}H_0^++H_n}{H_n}=(H_{n+1}/H_n)H_0^+.\]
This completes the proof.
\end{proof}
Now we are able to determine the coradical filtration of $\theta(H)$ by using the following lemma.
\begin{lemma}\label{gr}
Let $C$ be a coalgebra with a coalgebra filtration $\{F_n\}$ such that $F_0=C_0$. If the associated graded coalgebra with respect to $\{F_n\}$ is coradically graded, then $\{F_n\}$ agrees with the coradical filtration of $C$.
\end{lemma}
\begin{proof}
Denote the associated graded coalgebra with respect to $\{F_n\}$ by $\gr _FC$.
By the definition of coalgebra filtration and the fact that $F_0=C_0$, it is easy to see that $F_n\subset C_n$ for any $n\ge1$. If the assertion is not true, then choose $n$ to be minimal such that $F_n\subsetneq C_n$. Pick some element $y\in C_n\backslash F_n$. Then 
\begin{equation}\label{com}
\Delta(y)\in \sum\limits_{i=0}^{n}C_i\otimes C_{n-i}= C_n\otimes F_0+F_0\otimes C_n+\sum\limits_{i=1}^{n-1}F_i\otimes F_{n-i}.
\end{equation}
Suppose $y\in F_m\backslash F_{m-1}$ for some $m\ge n+1$. Let $\overline{y}$ be the corresponding non-zero element in $\gr_FC(m)$. Since $\gr _FC$ is coradically graded, $\overline{y}$ is not in $(\gr _FC)_{m-1}$. But by $(\ref{com})$, $\overline{y}$ is in $(\gr _FC)_{1}$, which is a contradiction. This completes the proof.
\end{proof}

\begin{proposition}\label{piH}
The coradical filtration of the coalgebra $\theta(H)$ is $\{\pi_H(H_n)\}_{n\ge 0}$.
\end{proposition}
\begin{proof}
By Proposition \ref{Filtration}, $\gr_F\theta(H)\cong \theta(\gr H)$ as coalgebras. By the proof of \cite[Theorem 3]{R}, $\theta(\gr H)\cong R$ as graded coalgebras, where $R$ is defined in $(\ref{bidecomp})$. By \cite[p.15]{AS}, $R$ is coradically graded. The result now follows from Lemma \ref{gr}.
\end{proof}

\section{GK-dimensions of $H$ and $\gr H$}

This section is devoted to the proof of Theorem \ref{equality}. Let $H$ be a pointed Hopf algebra with group-like elements $G$. As mentioned in the previous section, $\gr H$ has a decomposition
$
\gr H\cong R\#H_0$. By construction, the graded algebra $R$ is connected in the sense that $R(0)=k$. So if $R$ is finitely generated as an algebra, then it is locally finite.

\begin{lemma}\label{socole}
Retain the above notation. Let $D$ be a graded subcoalgebra of $\gr H$. Then 
\[D\subset \bigoplus_{i\ge0}\bigoplus_{h\in G(D)}R(i)h. \]
\end{lemma}
\begin{proof}
Let $y$ be a non-zero element in $D$. By assumption we can further assume that $y$ is homogeneous of degree $s$. If $s=0$, then $y\in D(0)=kG(D)$. If $s\ge 1$, then by the decomposition (\ref{bidecomp}), we can write $y=\sum_{i=1}^{N}y_ih_i$ where $h_i$'s are distinct group-like elements and $0\neq y_i\in R(n)$. We only need to show that $h_i\in G(D)$ for any $i$. Let $\psi: \gr H\rightarrow H_0$ be the canonical Hopf projection. Then it is obvious that $\psi$ maps $D$ onto $D(0)=kG(D)$. Now we have
\[(Id\otimes \psi)\Delta(y)=\sum_{i=1}^{N}y_ih_i\otimes h_i\in D\otimes D(0).\]
Hence $h_i\in G(D)$ by the choice of $y_i$ and $h_i$.
\end{proof}

The next lemma is about the GK-dimensions of skew group algebras. It can be viewed as a generalization of \cite[Lemma 5.5]{Z}. Let $\Gamma$ be a group and $A$ an algebra with a left $G$-action. As $k$-spaces, the skew group algebra $A*\Gamma$ is isomorphic to $A\otimes k\Gamma$. The multiplication is given by 
\[(a*g)(b*h)=a(g.b)*gh,\]
where $a, b\in A$, $g, h\in \Gamma$ and $a*g$ stands for $a\otimes g$. We will omit the $*$ in $a*g$ if there is no confusion. We say that the $\Gamma$-action on $A$ is \textbf{locally finite} if any finite-dimensional subspace of $A$ is contained in a finite-dimensional $\Gamma$-submodule of $A$.

\begin{lemma}\label{addition}
Let $A$ and $\Gamma$ be as above and suppose that the $\Gamma$-action on $A$ is locally finite. Then 
\[\GK A*\Gamma=\GK A+ \GK k\Gamma.\]
\end{lemma}
\begin{proof} We say a subalgebra $B$ of $A$ is $\mathbf{\Gamma}$-\textbf{affine} if $B$ is generated as an algebra by a finite-dimensional $\Gamma$-submodule of $A$.
 It is easy to check by the local finiteness condition that 
\[\GK A=\sup\limits_B \GK B,\]
where $B$ runs over all $\Gamma$-affine subalgebras of $A$.
Next we claim that 
\[\GK A*\Gamma=\sup\limits_{B, L}\GK B*L,\]
where $B$ runs over all $\Gamma$-affine subalgebras of $A$ and $L$ runs over all finitely generated subgroups of $\Gamma$. In fact, by the definition of the GK-dimension, $\GK A*\Gamma=\sup\limits_{E} \GK E$, where the supremum is taken over all finitely generated subalgebras $E$ of $A*\Gamma$. Let $V$ be a finite-dimensional generating subspace of $E$. Then there exists $g_1, \cdots, g_s\in G$ and some finite-dimensional subspace $W$ of $A$ such that $V\subset Wg_1+Wg_2+\cdots Wg_s$. By the local finiteness condition we can further 
assume that $W$ is a finite-dimensional $\Gamma$-submodule of $A$. 

Let $B$ be the subalgebra of $A$ generated by $W$ and let $L$ be the subgroup of $G$ generated by $g_1, \cdots, g_s$ and their inverses. Then it is clear that $B$ is $\Gamma$-affine and $E\subset B*L$. This prove the claim.
Now, 
\begin{align*}
\GK A*\Gamma&=\sup\limits_{B, L}\GK B*L\\
&=\sup\limits_{B, L}( \GK B+\GK kL)\\
&=\sup\limits_B \GK B+\sup\limits_L\GK kL\\
&=\GK A+\GK k\Gamma.
\end{align*}
The second equality follows from \cite[Lemma 5.5]{Z}. For the last equality, one just notices that $\GK k\Gamma=\sup\limits_L\GK kL$ by Proposition \ref{localfinite}.
\end{proof}

Recall from \cite[\S 5.4]{Mont} that for $g\in G$, 
\[P_{1, g}(H)=\{x\in H \,|\,\Delta(x)=x\otimes1+g\otimes x\}.\]
Let $P'_{1, g}(H)$ be a subspace of $P_{1, g}(H)$ such that $P_{1, g}(H)=k(1-g)\oplus P'_{1, g}(H)$ and define
\[P'_T(H)=\bigoplus_{g\in G}P'_{1, g}(H).\]

\begin{lemma}\label{1finite}
Let $H$ be a pointed Hopf algebra. Then the following statements are equivalent.
\begin{itemize}
\item[\textup{(I)}] $\dim_k P'_T(H)< \infty$,
\item[\textup{(II)}] $\dim_k R(1)<\infty$,
\item[\textup{(III)}] the graded algebra $R$ is locally finite,
\item[\textup{(IV)}] the coalgebra $\theta(H)$ is coradically finite.
\end{itemize}
\end{lemma}
\begin{proof}
 By \cite[Theorem 5.4.1(1)]{Mont} and the definition of $R$, $P'_T(H)\cong R(1)$ as $k$-spaces. Hence \textup{(I)} and \textup{(II)} are equivalent. Since $R$ is a coradically graded coalgebra and $R(0)=k$, the equivalence of \textup{(II)} and \textup{(III)} follows from \cite[Lemma 2.3 (2)]{MS}. As shown in the proof of Theorem \ref{equality}, $\dim_k \pi_H(H_n)=\dim_k\sum_{i=1}^{n} R(i)$. By Proposition \ref{piH}, $\{\pi_H(H_n)\}_{n\ge0}$ is the coradical filtration of $\theta(H)$. This shows that \textup{(III)} and \textup{(IV)} are equivalent.
\end{proof}

Now we are ready to prove the main theorem of this section.
\begin{theorem}\label{equality}
Retain the above notation. Suppose that $\dim_k R(1)<\infty$. Then
\begin{equation}\label{ineq}
\GK R+ \GK kG=\GK \gr H\le\GK H\le\GK kG+\gamma,
\end{equation}
where $\gamma=\varlimsup\limits_{n\rightarrow \infty}\log_n \dim_k  \pi_H(H_n)=\varlimsup\limits_{n\rightarrow \infty}\log_n \dim_k \bigoplus_{i=0}^{n} R(i)$. If $R$ is a finitely generated algebra, then 
\begin{equation}\label{eq}
\GK R+ \GK kG=\GK \gr H=\GK H=\GK kG+\gamma,
\end{equation}
\end{theorem}
\begin{proof}
Let $V_n=\bigoplus_{i=1}^{n} R(i)$. By Lemma \ref{1finite}, $V_n$ is finite-dimensional for any $n$. Also, $\{V_n\}_{n\ge0}$ is the coradical filtration of $R$ since $R$ is a coradically graded coalgebra. On the other hand, $R\cong \theta(\gr H)$ as graded coalgebras. It then follows from Proposition \ref{Filtration} and Proposition \ref{piH} that $V_n\cong  \pi_H(H_n)$ as $k$-spaces.
 As a consequence, the number $\gamma$ is well defined.
 
It is well known that $R\#kG$ is just $R*G$ as algebras. Moreover, the $G$-action on $R$ preserves the grading. Since every finite-dimensional subspace $V$ of $R$ is contain in $V_s$ for some $s\ge0$, we see that the $G$-action on $R$ is locally finite. Now by Lemma \ref{addition}, $\GK R+ \GK kG=\GK \gr H$. By \cite[Lemma 6.5]{KL}, $\GK\gr H\le \GK H$. Next we are going to show that $\GK H\le\GK kG+\gamma$.

Now let $C$ be a finite-dimensional subspace of $H$. Without loss of generality, we can assume that $C$ is a subcoalgebra of $H$. Let $S=G(C)$. By the choice of $C$, the set $S$ is finite. Denote by $G_S(\ell)$ the set of elements in $G$ that can be expressed 
as products of $\le \ell$ elements in $S\cup S^{-1}$ and let $g_S(\ell)=|G_S(\ell)|$. By Corollary \ref{approx}, $G(C^\ell)\subset G_S(\ell)$.
Suppose that $C\subset H_N$ for some $N\ge 1$. Then $C^n\subset H_{nN}$. Let $D=\gr C^n$, the associated graded coalgebra of $C^n$ with respect to its coradical filtration. Notice that $D$ is naturally embedded in $\gr H$. Since $G(D)$ can be identified with $G(C^n)$, we have $G(D)\subset G_S(n)$.

Now by Lemma \ref{socole}, we have
\begin{align*}
D &\subset \bigoplus_{i=0}^{nN}\bigoplus_{h\in G(D)}R(i)h\\
&=\bigoplus_{h\in G(D)}V_{nN}h\subset \bigoplus_{h\in G_S(n)}V_{nN}h.
\end{align*}
As a consequence, 
\[\dim_k C^n=\dim_kD\le \dim_k V_{nN}\cdot g_S(n).\]
Therefore, 
\begin{align*}
\varlimsup_{n\rightarrow \infty}\log_n \dim_k C^n&\le\varlimsup_{n\rightarrow \infty}\log_n  \dim_k V_{nN}\cdot g_S(n)\\
&\le\varlimsup_{n\rightarrow \infty}\log_n  \dim_k V_{nN}+ \varlimsup_{n\rightarrow \infty}\log_ng_S(n)\\
&\le \varlimsup_{n\rightarrow \infty}\log_n  \dim_k V_{n}+ \varlimsup_{n\rightarrow \infty}\log_ng_S(n)\\
&\le \gamma+ \GK kG.
\end{align*}
This proves $(\ref{ineq})$.

When $R$ is finitely generated, by \cite[Proposition 6.6]{KL}, $\GK R=\varlimsup\limits_{n\rightarrow \infty}\log_n \dim_k  V_n=\gamma$. Combining this fact with $(\ref{ineq})$, we have $(\ref{eq})$.
\end{proof}

\begin{remark} In Theorem \ref{equality}, if we further assume that $G$ is a finite group, then $\{H_n\}_{n\ge 0}$ is a finite filtration in the sense that $\dim_k H_n< \infty$ for any $n\ge 0$. This is true because $H_n/H_{n-1}\cong R(n)\otimes kG$ as $k$-spaces for all $n$. In this case, the result $\GK \gr H= \GK H$ follows from \cite[Proposition 6.6]{KL}.
\end{remark}

\begin{remark} It is easy to check that $\gr H$ is finitely generated if and only if both $R$ and $kG$ are finitely generated. Hence if $\gr H$ is finitely generated, then $\GK R+ \GK kG=\GK \gr H=\GK H$.
\end{remark}
If $H$ is a finitely generated pointed Hopf algebra, then $G$ is a finitely generated group by Corollary \ref{approx}. However, the finite generation of $H$ does not imply that $\dim_kR(1)<\infty$, as shown in the following example. 

\begin{example}
Let the base field $k$ be $\mathbb{F}_p$, where $p$ is a prime number. Let $H=k[x]$. Then $H$ is a connected Hopf algebra where $x$ is primitive. It is well known that $\gr H\cong k[x_1, x_2, \cdots]/(x^p_1, x^p_2, \cdots)$ with $x_i$ being primitive. As a consequence, $\GK \gr H=0$ since every finitely generated subalgebra of $\gr H$ is finite-dimensional. On the other hand, $\GK H=1$. 
\end{example}
The above example relies heavily on the assumption that the base field $k$ has characteristic $p$. In fact, based on known examples, it is conjectured that if the base field is of characteristic $0$, and $H$ is finitely generated with finite GK-dimension, then $R(1)$ is always finite-dimensional. Some partial results are discussed in \cite[Section 3]{WZZ2}.

We conclude this section with a straightforward corollary. For the definitions and basic properties of Yetter-Drinfeld modules and Nichols algebras, a good reference is \cite{AS}.
\begin{corollary}
Let $H$ be a pointed Hopf algebra with group-like elements $G$. If $\gr H\cong \mathcal{B}(V)\#kG$, where $V$ is a finite-dimensional left Yetter-Drinfeld module over $G$ and $\mathcal{B}(V)$ is the Nichols algebra of $V$, then 
\[\GK \mathcal{B}(V)+ \GK kG=\GK \gr H= \GK H.\]
\end{corollary}

\section{Connected Hopf algebras}
This section is primarily devoted to the study of connected Hopf algebras. Let $H$ be a connected Hopf algebra. Then its associated graded Hopf algebra $\gr H$ with respect to the coradical filtration is also connected. Moreover, the natural grading on $\gr H$ makes it into a coradically graded Hopf algebra as mentioned in Section \ref{pre}. In fact, we are able to show that if the base field is algebraically closed of characteristic $0$ and $\GK H<\infty$, then $\GK H$ must be a non-negative integer $\ell$ and $\gr H$ is isomorphic to the polynomial ring in $\ell$ variables as algebras (see Proposition \ref{polynomial} and Theorem \ref{integer}). As a consequence, we derive some ring-theoretic properties of such Hopf algebras. For instance, we show that they are always domains, which reproves an unpublished result by Le Bruyn.

By \cite[Lemma 5.2.10]{Mont}, a connected bialgebra is automatically a connected Hopf algebra. Furthermore, we have the following lemma.
\begin{lemma}\label{sub}
Let $H$ be a connected Hopf algebra and $K$ a sub-bialgebra of $H$. Then $K$ is a Hopf subalgebra of $H$. 
\end{lemma}
\begin{proof}
Let $S$ be the antipode of $H$. We need to show that $S(K)\subset K$. It suffices to show that $S(K_n)\subset K$ for any $n\ge 0$. When $n=0$, the statement is true since $K_0$ is spanned by the unit $1$. For any $n\ge 1$ and $c\in K_n$, by \cite[Lemma 5.3.2]{Mont}, $\Delta(c)=1\otimes c+c\otimes 1+ \sum a_i\otimes b_i$, where $a_i, b_i\in K_{n-1}$. 
Since $S$ is the convolution inverse of the identity map, we have $S(c)+c+\sum a_iS(b_i)=\e(c)$.
By induction hypothesis, $S(c)\in K$. This completes the proof.
\end{proof}

The following technical lemma will be used frequently in the rest of the paper.
\begin{lemma}\label{regular}
Let $f: A \rightarrow B$ be a surjective algebra map. If $A$ is a Noetherian prime algebra and $\GK A< \GK B+1<\infty$, then $f$ is an isomorphism.
\end{lemma}
\begin{proof}
We only have to show that $I:=\ker f$ is zero. If not, then by Goldie's theorem, the ideal $I$ contains a regular element. Now by \cite[Proposition 3.15]{KL}, $\GK B+1\le \GK A$. But this is a contradiction.
\end{proof}

\begin{lemma}\label{commutative}
Let $K=\bigoplus_{n=0}^{\infty}K(n)$ be a graded Hopf algebra with $K(0)=k$. Then the following statements are true.
\begin{itemize}
\item[\textup{(I)}] If $K$ is generated in degree one, then $K$ is cocommutative;
\item[\textup{(II)}] If $K$ is coradically graded, then $K$ is commutative.
\end{itemize}
\end{lemma}
\begin{proof}
It is easy to check that if a Hopf algebra is generated by elements $x$ such that $\Delta(x)=\tau\Delta(x)$, where $\tau$ is the twisting map, then the Hopf algebra is cocommutative. Since $K(1)$ is spanned by primitive elements, the statement $\textup{(I)}$ is true.

For the second statement, 
by Remark \ref{gradedaffine}, we can assume, without loss of generality, that $K$ is finitely generated and thus locally finite. Let $S=\bigoplus_{n=0}^{\infty}K(n)^*$ be the graded dual of $K$. Then $S$ is also a graded Hopf algebra with $S(0)=k$. By \cite[Lemma 5.5]{AS2}, $S$ is generated in degree one and thus cocommutative by $\textup{(I)}$. Hence $K$ is commutative.
\end{proof}
Notice that for any connected Hopf algebra $H$, $\gr H$ is connected coradically graded. Hence the following proposition is clear.
\begin{proposition}
Let $H$ be a connected Hopf algebra. Then $\gr H$ is commutative.
\end{proposition}

In fact, if the base field is algebraically closed of characteristic $0$, we can say more about the algebra structure of a connected coradically graded Hopf algebra.

\begin{proposition}\label{polynomial}
Let $K=\bigoplus_{n=0}^{\infty}K(n)$ be a coradically graded Hopf algebra with $K(0)=k$ and assume that the base field $k$ is algebraically closed of characteristic $0$. If $K$ is finitely generated, then $K$ is isomorphic to the polynomial ring in $\ell$ variables for some $\ell\ge0$ as algebras.

\end{proposition}
\begin{proof}
Since $K$ is finitely generated commutative, $K\cong \mathcal{O}(\Gamma)$, the coordinate ring of some algebraic group $\Gamma$ over $k$. Hence $K$ has finite global dimension. Now the result follows from \cite[III.2.5]{NO}.
\end{proof}

The previous proposition leads to the following theorem, which is a result by Le Bruyn (unpublished).
\begin{theorem}\label{domain}
Assume that the base field $k$ is algebraically closed of characteristic $0$. Let $H$ be a connected Hopf algebra. Then $H$ is a domain.
\end{theorem}
\begin{proof}
We only need to show that $\gr H$ is a domain. By Remark \ref{gradedaffine}, every finite subset of $\gr H$ is contained in a finitely generated graded Hopf subalgebra of $\gr H$. By Proposition \ref{polynomial}, such subalgebras are domains. The result then follows.
\end{proof}

\begin{remark}
In Theorem \ref{domain}, the statement fails if the base field is of characteristic $p$. For example, let $k=\mathbb{F}_p$ and $H=k[x]/(x^p)$. It is well known that $H$ has a unique connected Hopf algebra structure under which $x$ is primitive. Obviously, $H$ is not a domain.
\end{remark}
\begin{lemma}\label{GKplus1}
Assume that the base field $k$ is algebraically closed of characteristic $0$. Let $K$ be a connected coradically graded Hopf algebra and $L$ a finitely generated graded Hopf subalgebra of $K$. If $L\neq K$, then $\GK K\ge \GK L+1$.
\end{lemma}
\begin{proof}
Let $N$ be the smallest number such that $L(N)\neq K(N)$. Pick $y\in K(N)\setminus L(N)$. By the choice of $N$ we see that $\Delta(y)=1\otimes y+y\otimes 1+w$, where $w\in \bigoplus_{i=1}^{N-1}L(i)\otimes L(N-i)$. Hence the algebra $P$ generated by $L$ and $y$ is a finitely generated graded sub-bialgebra of $K$. By Lemma \ref{sub}, $P$ is a Hopf subalgebra of $K$. By replacing $K$ with $P$, we may assume that $K$ is also finitely generated. Now we have an exact sequence of Hopf algebras
\[0\rightarrow L\rightarrow K\rightarrow H\rightarrow 0,\]
in the sense that $L$ is a normal Hopf subalgebra of $K$ and $H=K/L^+K=K/KL^+$. Since $L\neq K$, the connected Hopf algebra $H$ is not isomorphic to $k$ by \cite[Theorem 4.3]{T1} and therefore $\GK H\ge1$. By taking the spectrum, we get an exact sequence of algebraic groups \cite[Theorem 5.2]{T1}
\[1\rightarrow \Gamma_1\rightarrow \Gamma_2\rightarrow \Gamma_3\rightarrow 1.\]
By \cite[7.4 Proposition B]{H}, $\dim \Gamma_2=\dim \Gamma_1+\dim \Gamma_3$, where $\dim$ represents the dimension of an affine variety. It is well known that for any affine variety $X$, $\dim X=\GK \mathcal{O}(X)$. Consequently,
\[\GK K=\GK L+\GK H.\]
Now the result follows since $\GK H\ge1$.
\end{proof}
In next section, we will lift the result to the ungraded case in Lemma \ref{GKcrit}.

\begin{lemma}\label{finiteGK}
Assume that the base field $k$ is algebraically closed of characteristic $0$. Let $K$ be a connected coradically graded Hopf algebra. Then $K$ has finite GK-dimension if and only if $K$ is finitely generated.
\end{lemma}
\begin{proof}
If $K$ is finitely generated, then by Proposition \ref{polynomial}, $K$ has finite GK-dimension. Now assume that $K$ is not finitely generated. If $\dim_kK(1)=\infty$, then $K$ has a Hopf subalgebra isomorphic to $U(\mathfrak{g})$, where $\mathfrak{g}:=K(1)$ is an infinite-dimensional Lie algebra. Hence $\GK K=\infty$. Now assume that $\dim_k K(1)<\infty$. It suffices to show that there is a chain of Hopf subalgebras $K^{(1)}\subset K^{(2)}\subset\cdots$ such that $\GK K^{(i)}+1\le\GK K^{(i+1)}$. Let $K^{(1)}$ be the subalgebra generated by $K(1)$. Then $K^{(1)}$ is a finitely generated graded Hopf subalgebra of $K$. Since $K^{(1)}\subsetneq K$ by assumption, there is some homogeneous element $y\in K\setminus K^{(1)}$ such that $\Delta(y)=1\otimes y + y\otimes 1+w$ where $w\in (K^{(1)})^+\otimes (K^{(1)})^+$. Let $K^{(2)}$ be the subalgebra generated by $K^{(1)}$ and $y$. It is obvious that $K^{(2)}$ is again a finitely generated graded Hopf subalgebra. By Lemma \ref{GKplus1}, $\GK K^{(2)}\ge \GK K^{(1)}+1$. Now $K^{(2)}\subsetneq K$, so we can repeat the above process and get the desired chain of Hopf subalgebras. This completes the proof.
\end{proof}

Now we are able to deliver the following theorem.
\begin{theorem}\label{integer}
Assume that the base field $k$ is algebraically closed of characteristic $0$ and let $H$ be a connected Hopf algebra. Then the following statements are equivalent:
\begin{enumerate}
\item[\textup{(I)}] $\GK H< \infty$;
\item[\textup{(II)}] $\GK \gr H< \infty$;
\item[\textup{(III)}] $\gr H$ is finitely generated;
\item[\textup{(IV)}] $\gr H$ is isomorphic to the polynomial ring of $\ell$ variables for some $\ell\ge0$ as algebras.
\end{enumerate}
In this case, $\GK H=\GK \gr H$, which is a positive integer.
\end{theorem}
\begin{proof}
If $\GK H=\infty$, we need to show that $\GK \gr H$ is also infinity. If not, by Lemma \ref{finiteGK}, $\gr H$ is finitely generated. Then by Theorem \ref{equality} (or \cite[Proposition 6.6]{KL}), $\GK H=\GK \gr H< \infty$, which is a contradiction. If $\GK H<\infty$, by \cite[Lemma 6.5]{KL}, $\GK \gr H\le \GK H<\infty$. Hence $\textup{(I)}$ and $\textup{(II)}$ are equivalent. The equivalence of $\textup{(II)}$ and $\textup{(III)}$ is just Lemma \ref{finiteGK}. The equivalence  of $\textup{(III)}$ and $\textup{(IV)}$ follows from Proposition \ref{polynomial}.

If one of the four conditions holds, then $\GK H=\GK \gr H$ by Theorem \ref{equality}. Moreover, in this case $\GK \gr H$ is a positive integer by $\textup{(IV)}$. This completes the proof.
\end{proof}

As a consequence of Theorem \ref{integer}, a connected Hopf algebra enjoys many nice ring-theoretical properties. A few of them are listed in the following corollary.

\begin{corollary}
Assume that the base field $k$ is algebraically closed of characteristic $0$ and let $H$ be a connected Hopf algebra of GK-dimension $\ell<\infty$. Then $H$ is 
\begin{itemize}
\item[\textup{(I)}] a noetherian domain of global dimension $\ell$ and Krull dimension $\le \ell$;
\item[\textup{(II)}] Auslander-regular;
\item[\textup{(III)}] GK-Cohen-Macaulay, i.e., for any non-zero finitely generated $H$-module $M$, 
\[j(M)+\GK M=\GK H,\]
where $j(M):=\min\{n \,|\,\Ext^n_H(M, H)\neq 0\}$.
\end{itemize}
\begin{proof}
By Theorem \ref{integer}, $\gr H$ is a noetherian domain of global dimension and Krull dimension $\ell$. Now by \cite[Lemma I.12.12, Theorem I.12.13]{BG} and \cite[Lemma 5.6, Corollary 6.18]{MR}, $H$ is a noetherian domain of global dimension and Krull dimension $\le \ell$. Moreover, by taking $M$ to be the trivial $H$-module $k$ in $(\textup{III})$, we have $j(k)=\ell$. This shows that the global dimension of $H$ is $\ell$.

Since $\gr H$ is noetherian, the filtration $\{H_n\}_{n\ge 0}$ on $H$ is Zariskian by \cite[2.10]{Bj}. Then the statement $\textup{(II)}$ follows from \cite[Theorem 3.9]{Bj}.

For the statement $\textup{(III)}$, we first choose a good filtration $\{M_n\}_{n\in\mathbb{Z}}$ of $M$ in the sense of \cite[Definition 5.1]{LO}. It then follows from \cite[Lemma 5.4]{LO} that $\gr M$ is a finitely generated $\gr H$-module. It is clear that $\gr H$ is GK-Cohen-Macaulay. 
Hence $j_{\gr H}(\gr M)+\GK \gr M=\GK \gr H$. As mentioned in the proof of \cite[Theorem 3.9]{Bj}, $j_{\gr H}(\gr M)=j(M)$. By Theorem \ref{integer}, $\GK \gr H=\GK H$ and $\gr H$ is a finitely generated algebra. It then follows from \cite[Proposition 6.6]{KL} that $\GK \gr M=\GK M$. This completes the proof.
\end{proof}
\end{corollary}

\section{Connected Hopf algebras of GK-dimension three}
\textbf{Throughout this section, the base field $k$ is algebraically closed of characteristic zero.} We are going to classify all connected Hopf algebras of GK-dimension three. To begin with, we introduce two classes of Hopf algebras.

\begin{example}\label{typeA}
Let $A$ be the algebra generated by elements $X, Y, Z$ satisfying the following relations,
\begin{align*}
[X, Y]=0,\\
[Z, X]=\lambda_1 X+\alpha Y,\\
[Z, Y]=\lambda_2 Y,
\end{align*}
where $\alpha=0$ if $\lambda_1\neq \lambda_2$ and $\alpha=0$ or $1$ if $\lambda_1= \lambda_2$.
Then $A$ becomes a Hopf algebra via
\begin{align*}
\e(X)=0,\quad \Delta(X)=1\otimes X+ X\otimes 1,\quad S(X)=-X,\\
\e(Y)=0,\quad \Delta(Y)=1\otimes Y+ Y\otimes 1,\quad S(Y)=-Y,\\
\e(Z)=0,\quad \Delta(Z)=1\otimes Z+X\otimes Y+ Z\otimes 1,\quad S(Z)=-Z+XY.
\end{align*}
We denote this Hopf algebra by $A(\lambda_1, \lambda_2, \alpha)$.
\end{example}

\begin{example}\label{typeB}
 Let $B$ be the algebra generated by elements $X, Y, Z$ satisfying the following relations,
\begin{align*}
[X, Y]=Y,\\
[Z, X]=-Z+\lambda Y,\\
[Z, Y]=\frac{1}{2}Y^2,
\end{align*}
where $\lambda \in k$.
Then $B$ becomes a Hopf algebra via
\begin{align*}
\e(X)=0,\quad \Delta(X)=1\otimes X+ X\otimes 1,\quad S(X)=-X,\\
\e(Y)=0,\quad \Delta(Y)=1\otimes Y+ Y\otimes 1,\quad S(Y)=-Y,\\
\e(Z)=0,\quad \Delta(Z)=1\otimes Z+X\otimes Y+ Z\otimes 1,\quad S(Z)=-Z+XY.
\end{align*}
We denote this Hopf algebra by $B(\lambda)$.
\end{example}

\begin{proposition}
The algebras $A(\lambda_1, \lambda_2, \alpha)$ and $B(\lambda)$ are connected Hopf algebras of GK-dimension three.
\end{proposition}
\begin{proof}
We only prove the statement for $B(\lambda)$. The case of $A(\lambda_1, \lambda_2, \alpha)$ can be proved analogously. 

As mentioned in \cite[Section 1]{GZ}, to check $B(\lambda)$ is a Hopf algebra, it suffices to check the Hopf algebra axioms on a set of algebra generators for $B(\lambda)$, namely, $X, Y$ and $Z$. This is easy and we leave it to the readers.

By Bergman's Diamond Lemma, the algebra $B(\lambda)$ has a $k$-linear basis of monomials
\[\{X^{w_1}Y^{w_2}Z^{w_3}\},\]
where $w_i\in \mathbb{N}$. Define the degree of $X^{w_1}Y^{w_2}Z^{w_3}$ to be $w_1+w_2+2w_3$ and let $F_n$ be the $k$-space spanned by all monomials of degree $\le n$. It is easy to check that $\{F_n\}_{n\ge 0}$ is an algebra filtration on $A$ by the defining relations. Hence by \cite[Lemma 6.1 (b)]{KL}, 
\[\GK B(\lambda)=\varlimsup_{n\rightarrow \infty}\log_n\dim_k F_n=3.\]
Next, we claim that $\{F_n\}_{n\ge 0}$ is also a coalgebra filtration on $B$, i.e. $\Delta(F_n)\subset\sum_{i=0}^{n}F_i\otimes F_{n-i}$ for any $n$. Let $X^{w_1}Y^{w_2}Z^{w_3}$ be a monomial such that $w_1+w_2+2w_3\le n$. Then 
\begin{align*}
\Delta(X^{w_1}Y^{w_2}Z^{w_3})=&\Delta(X)^{w_1}\Delta(Y)^{w_2}\Delta(Z)^{w_3}\\
\in& (F_0\otimes F_1+F_1\otimes F_0)^{w_1+w_2}\cdot(\sum_{i=0}^2 F_i\otimes F_{2-i})^{w_3}\\
\subset &(\sum_{i=0}^{w_1+w_2} F_i\otimes F_{w_1+w_2-i})\cdot(\sum_{i=0}^{2w_3} F_i\otimes F_{2w_3-i})\\
\subset &\sum_{i=0}^{n}F_i\otimes F_{n-i}.
\end{align*}
For the last two inclusions, we use the fact that $\{F_n\}_{n\ge 0}$ is an algebra filtration. Then it follows from \cite[Lemma 5.3.4]{Mont} that the coradical of $B(\lambda)$ is contained in $F_0$, which is one-dimensional. Hence $B(\lambda)$ is a connected coalgebra. This completes the proof.
\end{proof}

Before moving on to study connected Hopf algebras of GK-dimension three, we still need a few lemmas.
\begin{lemma}\label{GKcrit}
Let $H$ be a connected Hopf $k$-algebra of finite GK-dimension and $K$ a Hopf subalgebra of $H$. If $\GK K=\GK H$, then $K=H$.
\end{lemma}
\begin{proof}
By \cite[Lemma 5.2.12]{Mont}, $\gr K$ is naturally embedded in $\gr H$ as a graded Hopf subalgebra. Also, by Theorem \ref{integer}, $\GK\gr K=\GK K=\GK H=\GK \gr H$ and they are all finitely generated. It suffices to show that $\gr K=\gr H$. If not, by Lemma \ref{GKplus1}, $\GK \gr H\ge \GK \gr K+1$, which is a contradiction.
\end{proof}

The following proposition is a direct consequence of Lemma \ref{GKcrit}.
\begin{proposition}\label{enveloping}
Let $H$ be a connected Hopf algebra of finite GK-dimension. Then $\GK H\ge \dim_kP(H)$. If $\GK H=\dim_k P(H)$, then $H\cong U(\mathfrak{g})$ as Hopf algebras, where $\mathfrak{g}=P(H)$. If $\GK H=3$, then $\dim_k P(H)=2$ or $3$.
\end{proposition}
\begin{proof} Let $\mathfrak{g}=P(H)$. Then the injective Lie algebra map $\mathfrak{g}\hookrightarrow H$ induces a Hopf map $U(\mathfrak{g})\rightarrow H$. It is well known that $P(U(\mathfrak{g}))=\mathfrak{g}$ and therefore by \cite[Corollary 5.4.7]{Mont} the Hopf map $U(\mathfrak{g})\rightarrow H$ is injective. By PBW Theorem, $\GK U(\mathfrak{g})=\dim_k \mathfrak{g}$. Hence $\GK H\ge \GK U(\mathfrak{g})= \dim_kP(H)$. If $\GK H=\dim_k P(H)$, then we have $\GK H= \GK U(\mathfrak{g})$. Hence $H=U(\mathfrak{g})$ by Lemma \ref{GKcrit}. The last statement is from \cite[Lemma 5.11]{Z}.
\end{proof}
The following proposition is an easy consequence of Proposition \ref{enveloping}. It is also mentioned in \cite{GZ}.
\begin{proposition}\label{lessthan3}
Let $H$ be a connected Hopf algebra of GK-dimension strictly less than $3$. Then $\GK H= 0, 1$ or $2$. In fact, 
\begin{enumerate}
\item[\textup{(I)}] if $\GK H=0$, then $H\cong k$, the trivial Hopf algebra;
\item[\textup{(II)}] if $\GK H=1$, then $H\cong k[x]$ with $x$ being primitive;
\item[\textup{(III)}] if $\GK H=2$, then $H\cong U(\mathfrak{g})$, where $\mathfrak{g}$ is either the $2$-dimensional abelian Lie algebra  or the Lie algebra with basis $\{x, y\}$ and $[x, y]=y$.
\end{enumerate}
\end{proposition}
 
Now we focus on connected Hopf algebras of GK-dimension three.
The following theorem is the key to our main theorem of this section.
\begin{theorem}\label{3generators}
Let $H$ be a connected Hopf algebra of GK-dimension $\ge 3$ such that $\dim_k P(H)=2$. Then for any linearly independent primitive elements $x, y$, there exists $z\in H$ such that $\Delta(z)=1\otimes z+x\otimes y+z\otimes 1$. Moreover, if in addition $\GK H=3$, then for any such $z$, the set $\{x, y, z\}$ generates $H$ as an algebra.
\end{theorem}
We postpone the proof to the last section.
Now we are ready to deliver the main theorem of this section.
\begin{theorem}\label{classification}
Let $H$ be a connected Hopf algebra of GK-dimension three. Then $H$ is isomorphic to one of the following:
\begin{enumerate}
\item[\textup{(I)}] The enveloping algebra $U(\mathfrak{g})$ for some three-dimensional Lie algebra $\mathfrak{g}$;
\item[\textup{(II)}] The Hopf algebras $A(0, 0, 0)$, $A(0, 0, 1)$, $A(1, 1, 1)$ or $A(1, \lambda, 0)$ from Example \ref{typeA} for some $\lambda\in k$;

\item[\textup{(III)}] The Hopf algebras $B(\lambda)$ from Example \ref{typeB} for some $\lambda\in k$.
\end{enumerate}
\end{theorem}
\begin{proof}
By Proposition \ref{enveloping}, $\dim_k P(H)$ is either $2$ or $3$. If $\dim_k P(H)=3$, then by Proposition $\ref{enveloping}$, $H\cong U(\mathfrak{g})$ as Hopf algebras, where $\mathfrak{g}=P(H)$. This gives the Hopf algebras in $\textup{(I)}$. Now we focus on the case $\dim_k P(H)= 2$.

Let $\mathfrak{h}=P(H)$. Then $\mathfrak{h}$ is a two-dimensional Lie algebra. It is well known that there are two isomorphic classes of two-dimensional Lie algebras. 

Case 1: The Lie algebra $\mathfrak{h}$ is spanned by $x$ and $y$ with $[x, y]=0$. By Theorem \ref{3generators}, there is $z\in H$ such that $\Delta(z)=1\otimes z+x\otimes y+z\otimes 1$ and $x, y, z$ generate $H$ as an algebra. It is easy to check that $[z, x]$ and $[z, y]$ are primitive elements. Therefore
\begin{align}\label{relation1}
[z, x]=a_{11}x+ a_{12}y,\\
[z, y]=a_{21}x+ a_{22}y,\nonumber
\end{align}
where $a_{ij}\in k$.

Let $P$ be a $2\times2$ invertible matrix such that $P^{-1}AP$ is a Jordan matrix, where $A=(a_{ij})$. We take $\det P=1$. Let $P=(b_{ij})$ and $P^{-1}=(c_{ij})$. Then by setting $x'=b_{11}x+b_{21}y$ and $y'=b_{12}x+b_{22}y$, the relations $(\ref{relation1})$ become
\begin{align*}
[z, x']&=\lambda_1x'+ \alpha y',\\
[z, y']&=\lambda_2y',\nonumber
\end{align*}
where $\left({\begin{array}{cc}\lambda_1 &\alpha\\ 0 &\lambda_2\end{array}}\right)$ is a Jordan matrix. Now we have 
\[\Delta(z)=1\otimes z+(c_{11}x'+c_{21}y')\otimes(c_{12}x'+c_{22}y')+z\otimes 1.\]
Let $z'=z-\frac{1}{2}c_{11}c_{12}{x'}^2-\frac{1}{2}c_{11}c_{22}{y'}^2-c_{12}c_{21}x'y'$. Then a direct calculation shows that 
\begin{equation}\label{comul}
\Delta(z')=1\otimes z' +x'\otimes y'+z'\otimes 1,
\end{equation}
and
\begin{align}\label{relation2}
[z', x']&=\lambda_1x'+ \alpha y',\\
[z', y']&=\lambda_2y',\nonumber
\end{align}
Notice that $x', y', z'$ generate $H$ and $[x', y']=0$. 

If $\lambda_1=\lambda_2=0$ and $\alpha=0$ (resp, $\alpha=1$), then there is a surjective Hopf map from $A(0, 0, 0)$ (resp. $A(0, 0, 1)$) to $H$ sending $X, Y, Z$ to $x', y', z'$, respectively.

If $\lambda_1=\lambda_2\neq 0$ and $\alpha=1$, then there is a surjective Hopf map from $A(1, 1, 1)$ to $H$ sending $X, Y, Z$ to $x', \frac{1}{\lambda_1}y', \frac{1}{\lambda_1}z'$, respectively.

If $\lambda_1\neq 0$ and $\alpha=0$, then there is a surjective Hopf map from $A(1, \lambda_2/\lambda_1, 0)$ to $H$ sending $X, Y, Z$ to $\frac{1}{\lambda_1}x', y', \frac{1}{\lambda_1}z'$, respectively.

If $\lambda_2\neq 0$ and $\alpha=0$, then there is a surjective Hopf map from $A(1, \lambda_1/\lambda_2, 0)$ to $H$ sending $X, Y, Z$ to $\frac{1}{\lambda_2}y', -x', \frac{1}{\lambda_1}(z'-x'y')$, respectively.

By Lemma \ref{regular}, all the above surjective Hopf maps are isomorphisms. This completes the proof of $\textup{(II)}$.

Case 2: The Lie algebra $\mathfrak{h}$ is spanned by $x$ and $y$ with $[x, y]=y$. Again by Theorem \ref{3generators}, there is $z\in H$ such that $\Delta(z)=1\otimes z+x\otimes y+z\otimes 1$ and $x, y, z$ generate $H$ as an algebra. A straight calculation shows that $[z, y]-\frac{1}{2}y^2$ and $[z, x]+z$ are primitive elements.
Therefore
\begin{align*}
[z, x]=-z+a_{11}x+ a_{12}y,\\
[z, y]=\frac{1}{2}y^2+a_{21}x+ a_{22}y,
\end{align*}
where $a_{ij}\in k$.  By replacing $z$ with $z-a_{11}x$, we can assume that $a_{11}=0$. We claim that $a_{21}=a_{22}=0$. Notice that the relations between $x, y, z$ can be rewritten as
\begin{align*}
yx&= xy-y,\\
zx&= xz-z+ a_{12}y,\\
zy&= yz+\frac{1}{2}y^2+a_{21}x+a_{22}y.
\end{align*}
By these relations, we have 
\[z(yx)=xyz+\frac{1}{2}xy^2+a_{21}x^2+a_{22}xy-2yz+(a_{12}-1)y^2-2a_{21}x-2a_{22}y.\]
On the other hand, 
\[(zy)x=xyz+\frac{1}{2}xy^2+a_{21}x^2+a_{22}xy-2yz+(a_{12}-1)y^2-a_{22}y.\]
Since $z(yx)=(zy)x$ by associativity, we have $-2a_{21}x-2a_{22}y=-a_{22}y$, which implies that $a_{21}=a_{22}=0$. Now it is clear that there is a surjective Hopf algebra map from $B(a_{12})$ to $H$ sending $X, Y, Z$ to $x, y, z$, respectively.
By Lemma \ref{regular}, this surjective map is an isomorphism. This completes the proof of $\textup{(III)}$.
\end{proof}

\begin{remark} It is clear from the proof of Theorem \ref{classification} that any Hopf algebra $H$ listed in $\textup{(II)}$ and $\textup{(III)}$ has GK-dimension $3$ and $\dim_k P(H)=2$.
\end{remark}
In fact, by following the same lines as in case $1$ and case $2$ in the proof of the previous theorem, we have the following proposition.
\begin{proposition}
Let $H$ be a connected Hopf algebra of GK-dimension $\ge 3$ such that $\dim_k P(H)=2$. Then for any linearly independent primitive elements $x, y$, there exists $z\in H$ such that $\Delta(z)=1\otimes z+x\otimes y+z\otimes 1$. Moreover, the algebra $C$ generated by $x, y, z$ is a Hopf subalgebra of $H$ and $\GK C=3$.
\end{proposition}

For the rest of this section, we are going to look closer at the Hopf algebras listed in Theorem \ref{classification} $\textup{(II)}$ and $\textup{(III)}$. 
\begin{lemma}\label{dim1}
Let $H$ be a connected Hopf algebra of GK-dimension three such that $\dim_k P(H)=2$ and let $x, y, z$ be a set of generators as described in Theorem \ref{3generators}. Denote by $K$ the Hopf subalgebra generated by $x$ and $y$. Then $H_2/K_2$ is spanned by the image of $z$.
\end{lemma}
\begin{proof}
It is clear that $z\in H_2\setminus K_2$. Hence we only have to show that $H_2/K_2$ is one-dimensional. Notice that $H_1=K_1$. Therefore it suffices to show that $\gr H(2)/\gr K(2)$ is one dimensional. By Theorem \ref{3generators}, $\overline{x}, \overline{y}\in \gr H(1)$ and $\overline{z}\in \gr H(2)$ generate $\gr H$. It is also clear that $\gr K$ is the Hopf subalgebra of $\gr H$ generated by $\overline{x}$ and $\overline{y}$. As a consequence, $\gr H(2)= k\overline{z}+(\gr H(1))^2= k\overline{z}+(\gr K(1))^2= k\overline{z}+\gr K(2)$. This completes the proof.
\end{proof}

For any Hopf algebra $H$, the commutator ideal $[H, H]$ is a Hopf ideal \cite[Lemma 3.7]{GZ}. We call $H/[H, H]$ the \textbf{abelianization} of $H$. For any $h\in H$, let $\ad(h)\in \End_k(H)$ be the linear map sending $u$ to $[h, u]$ for any $u\in H$.

\begin{proposition}\label{isoclass}
For any given $\lambda\in k$, $A(0, 0, 0)$, $A(0, 0, 1)$, $A(1, 1, 1)$ and $A(1, \lambda, 0)$ are pairwise non-isomorphic. Also, $A(1, \lambda, 0)\cong A(1, \gamma, 0)$ if any only if $\lambda=\gamma$ or $\lambda\gamma=1$.
\end{proposition}
\begin{proof}
We start with the first statement. It is clear from defining relations that the abelianizations of $A(0, 0, 0)$, $A(0, 0, 1)$ and $A(1, 1, 1)$ are $A(0, 0, 0)$, $k[X, Z]$ and $k[Z]$, respectively. And the abelianization of $A(1, \lambda, 0)$ is $k[Z]$ if $\lambda\neq 0$, and $k[Y, Z]$ if $\lambda=0$. Now to prove the first statement, we only have to show $A(0, 0, 1)\ncong A(1, 0, 0)$ and $A(1, 1, 1)\ncong A(1, \lambda, 0)$ for $\lambda\neq 0$.

Suppose that $f$ is a Hopf isomorphism from $H':=A(1, 1, 1)$ to $ H:=A(1, \lambda, 0)$. We label the canonical generators of $H'$ by $X', Y'$ and $Z'$. Let $K'$ (resp. $K$) be the Hopf subalgebra of $H'$ (resp. $H$) generated by $X', Y'$ (resp. $X, Y$). Then $f$ restricts to a Hopf isomorphism from $K'$ to $K$. As a consequence, $f$ induces a linear isomorphism from $H'_2/K'_2$ to $H_2/K_2$. This indicates that $f(Z')$ is of the form $aZ+u$ for some $a\in k^{\times}$ and $u\in K_2$. Now consider the maps
$$\ad(Z'): H'_1\rightarrow H'_1\,\, \text{and}\,\, \ad(f(Z')): H_1\rightarrow H_1. $$
Since $f$ restricts to a linear isomorphism from $H'_1$ to $H_1$ and $f\circ \ad(Z')=\ad(f(Z'))\circ f$, the two maps must have the same eigenvalues and the same number of independent eigenvectors. However, from the defining relations we see that $\ad(Z')$ has only one linearly independent eigenvector while $\ad(f(Z'))=\ad(aZ+u)$ has two. This shows that $A(1, 1, 1)\ncong A(1, \lambda, 0)$. If we replace $H'$ and $H$ by $A(0, 0, 1)$ and $A(1, 0, 0)$ respectively, then the above argument shows that $A(0, 0, 1)\ncong A(1, 0, 0)$. This completes the proof of the first statement.

Next, we proceed to prove the second statement. As mentioned before, the abelianization of $A(1, \lambda, 0)$ is $k$ if $\lambda\neq 0$, and $k[Y]$ if $\lambda=0$. Hence $A(1, 0, 0)\ncong A(1, \lambda, 0)$ for any $\lambda\neq 0$. Now assume $A(1, \lambda, 0)\cong A(1, \gamma, 0)$ and we have to show that either $\lambda=\gamma$ or $\lambda\gamma=1$. Repeat the argument in the second paragraph of the proof by taking $H'=A(1, \lambda, 0)$ and $H=A(1, \gamma, 0)$. It is easy to check by defining relations that $\ad Z'$ has eigenvalues $\{1, \lambda\}$ and  $\ad f(Z')$ has eigenvalues $\{a, a\gamma\}$. Since they have the same eigenvalues, we must have 
$$\begin{cases}
1=a\\
\lambda=a\gamma
\end{cases}
\text{or}\,\,\,\,
\begin{cases}
1=a\gamma\\
\lambda=a
\end{cases}.$$
Clearly, these imply that either $\lambda=\gamma$ or $\lambda\gamma=1$.

Conversely, we only have to show that if $\lambda\gamma=1$, then $A(1, \lambda, 0)\cong A(1, \gamma, 0)$. Label the canonical generators of $A(1, \lambda, 0)$ by $X', Y'$ and $Z'$. Then there is a surjective Hopf map from $A(1, \lambda, 0)$ to $A(1, \gamma, 0)$ sending $X', Y', Z'$ to $Y, -\lambda X, \lambda (Z-XY)$, respectively. This map is an isomorphism by Lemma \ref{regular}.
\end{proof}

\begin{proposition} $B(\lambda)\cong B(\gamma)$ if and only if $\lambda=\gamma$.
\end{proposition}
\begin{proof}
Label the canonical generators of $B(\lambda)$ by $X', Y', Z'$. Suppose that $f$ is an isomorphism from $H':=B(\lambda)$ to $H:=B(\gamma)$. By Theorem \ref{classification}, $P(H')$ (resp. $P(H)$) are spanned by $X', Y'$ (resp. $X, Y$). Notice that $f$ restricts to a linear isomorphism from $P(H')$ to $P(H)$. Hence 
\begin{align*}
f(X')&=a_{11}X+a_{12}Y,\\
f(Y')&=a_{21}X+a_{22}Y,
\end{align*}
for some non-degenerate matrix $(a_{ij})$. Since $[f(X'), f(Y')]=f(Y')$, we have 
\[[a_{11}X+a_{12}Y, a_{21}X+a_{22}Y]=a_{21}X+a_{22}Y.\]
By using the fact $[X, Y]=Y$ and comparing the coefficients, we find that $a_{11}=1$, $a_{21}=0$ and $a_{22}\ne 0$. Since $f$ is a coalgebra map, 
\[\Delta f(Z')=(f\otimes f)\Delta(Z')=1\otimes f(Z')+(X+a_{12}Y)\otimes a_{22}Y+f(Z')\otimes 1.\]
Then it is easy to check that $f(Z')-a_{22}Z-\frac{1}{2}a_{12}a_{22}Y^2\in P(H)$. As a consequence, there exists $c, d\in k$ such that 
\[f(Z')=a_{22}Z+\frac{1}{2}a_{12}a_{22}Y^2+cX+dY.\]
Since $[f(Z'), f(Y')]=\frac{1}{2}f(Y')^2$, 
\[[a_{22}Z+\frac{1}{2}a_{12}a_{22}Y^2+cX+dY, a_{22}Y]=\frac{1}{2}a_{22}^2Y^2.\]
By comparing the coefficients, we find that $c=0$. Now the relation $[f(Z'), f(X')]=-f(Z')+\lambda f(Y')$ gives
\begin{equation}\label{compare}
[a_{22}Z+\frac{1}{2}a_{12}a_{22}Y^2+dY, X+a_{12}Y]=-a_{22}Z-\frac{1}{2}a_{12}a_{22}Y^2-dY+\lambda a_{22}Y.
\end{equation}
The left-hand side of (\ref{compare}) becomes 
\[-a_{22}Z+a_{22}\gamma Y+\frac{1}{2}a_{12}a_{22}Y^2-a_{12}a_{22}Y^2-dY.\]
Comparing this with the right-hand side of (\ref{compare}) we have $\lambda=\gamma$. This completes the proof.
\end{proof}
We conclude the section by two propositions regarding the algebra structures of the Hopf algebras $A(\lambda_1, \lambda_2, \alpha)$ and $B(\lambda)$, the first of which suggests that they can be considered as coalgebra deformations of universal enveloping algebras. However, we will not pursue this direction further.

\begin{proposition} For any choice of $(\lambda_1, \lambda_2, \alpha)$ (resp. $\lambda$), as an algebra, $A(\lambda_1, \lambda_2, \alpha)$ (resp. $B(\lambda)$) is isomorphic to the enveloping algebra of a solvable Lie algebra.
\end{proposition}
\begin{proof} For $A(\lambda_1, \lambda_2, \alpha)$, by the defining relations in Example \ref{typeA}, we have $A(\lambda_1, \lambda_2, \alpha)\cong U(\frak{g})$ as algebras where $\frak{g}$ is the  solvable Lie algebra spanned by $X, Y$ and $Z$. For $B(\lambda)$, let $Z':= Z-\frac{1}{2}XY$, then $B(\lambda)$ is generated by $X, Y$ and $Z'$ with the following relations
\begin{align*}
[X, Y]=Y,\\
[Z', X]=-Z'+\lambda Y,\\
[Z', Y]=0.
\end{align*}
Now it is clear that $B(\lambda)\cong U(\frak{g})$ where $\frak{g}$ is the solvable Lie algebra spanned by $X, Y$ and $Z'$.
\end{proof}

Since the Lie algebra $U(sl_2)$ is not solvable, we have the following corollary, which suggests  that $U(sl_2)$ has no non-trivial coalgebra deformations.
\begin{corollary}
 For any choice of $(\lambda_1, \lambda_2, \alpha)$ (resp. $\lambda$), the Hopf algebra $A(\lambda_1, \lambda_2, \alpha)$ (resp. $B(\lambda)$) is not isomorphic to $U(sl_2)$ as an algebra.
\end{corollary}

\section{Proof of Theorem \ref{3generators}}
This section is devoted to the proof of Theorem \ref{3generators}. The proof uses the cohomology of coalgebras, which we will recall briefly.

Let $C$ be a coaugmented coalgebra in the sense that there is a coalgebra map from the trivial coalgebra $k$ to $C$. Let $J=C^+$, the kernel of the counit, and one defines the reduced comultiplication on $J$ by
\[\overline{\Delta}(c)=\Delta(c)-(1\otimes c+c\otimes 1).\]
Then the \textbf{cobar construction} $\Omega C$ on $C$ is the differential graded algebra defined as follows:
\begin{itemize}
\item As a graded algebra, $\Omega C$ is the tensor algebra $T(J)$,
\item The differential in $\Omega C$ is given by
\begin{equation*}
\partial_C^n=\sum_{i=0}^{n-1} (-1)^{i+1} 
Id^{\otimes i}\otimes \overline{\Delta} \otimes Id^{\otimes (n-i-1)}.
\end{equation*}
\end{itemize}

Dually, given an augmented algebra $A$, one can construct a differential graded coalgebra $BA$, which is called the \textbf{bar construction} of $A$. See \cite[\S 19]{FHT} for basic properties of cobar and bar constructions.

\begin{lemma}\label{cohomology}
Let $C=U(\mathfrak{h})$ where $\mathfrak{h}$ is a two-dimensional Lie algebra spanned by $x$ and $y$. Then $\dim_k \h^2(\Omega C)=1$ and in fact $\h^2(\Omega C)=(x\otimes y)$ where $(x\otimes y)$ is the cohomology class defined by the cocycle $x\otimes y$.
\end{lemma}
\begin{proof}
By the PBW Theorem, the coalgebra $C$ has a basis of the form $\{x^iy^j \,|\, i, j\in \mathbb{N}\}$. It is also well known that $C$ becomes a graded coalgebra by setting $\deg x^iy^j =i+j$. Denote the $n$-th homogeneous part by $C(n)$. Now $J$ can be naturally identified with $\bigoplus_{i=1}^{\infty}C(i)$. Moreover, the graded $k$-linear dual of the graded coalgebra $C$ is isomorphic to $A=k[x_1, x_2]$ as graded algebras, where $x_i$ has degree $1$. By \cite[Lemma 8.6 (c)]{LPWZ2}, $B^\#A\cong \Omega C$ as DG algebras, where $B^\#A$ is the graded dual of the bar construction of $A$. On the other hand, by \cite[Lemma 4.2]{LPWZ}, $\h^\bullet(B^\#A)\cong \Ext^\bullet_A(k_A, k_A)$. As a consequence, $\dim_k \h^2(\Omega C)=\dim_k\Ext^2_A(k_A, k_A)=1$.

It is easy to check by definition that $x\otimes y$ is a cocycle, i.e. $\partial ^2(x\otimes y)=0$. We only have to show that $x\otimes y\notin \im\partial^1$. Suppose to the contrary that there is some $w\in C$ such that $\partial^1(w)=x\otimes y\in C(1)\otimes C(1)$. Then by a degree argument, the element $w$ is in $C(2)$, i.e. $w$ must be of the form $ax^2+ bxy+ cy^2$ for some $a, b, c\in k$. However, an easy calculation shows that $\partial^1(ax^2+ bxy+ cy^2)$ is in the $k$-subspace $V$ spanned by $x\otimes x$, $y\otimes y$ and $x\otimes y+y\otimes x$ and clearly $x\otimes y$ is not in $V$. This completes the proof.
\end{proof}

Now we are ready to prove Theorem \ref{3generators}.
\begin{proof}[of Theorem \ref{3generators}]
Let $C$ be the subalgebra of $H$ generated by $x$ and $y$. Then $C$ is a Hopf subalgebra of $H$ and $C$ is isomorphic to $U(\mathfrak{h})$ where $\mathfrak{h}$ is a two-dimensional Lie algebra.
Notice that by construction $C_1=H_1$. Let $N\ge 2$ be the least number such that $C_N\subsetneq H_N$. By \cite[Lemma 5.3.2]{Mont}, there exists $z'\in H_N\setminus C_N$ such that $\Delta(z')=1\otimes z'+z'\otimes 1+u$, where $u\in H_{N-1}\otimes H_{N-1}=C_{N-1}\otimes C_{N-1}\subset C\otimes C$. Without loss of generality, we assume that $\e(z')=0$.

Now we have two DG algebras, $(\Omega H, \partial_H)$ and $(\Omega C, \partial_C)$. In fact, $(\Omega C, \partial_C)$ can be viewed as a sub-complex of $(\Omega H, \partial_H)$. Notice that $0=\partial_H^2\partial_H^1(z')=\partial_H^2(u)=\partial_C^2(u)$, i.e. $u$ is a cocycle. We claim that $u$ represents a non-zero cohomology class in $\h^2(\Omega C)$. If not, there is $w\in C$ such that $\partial_C^1(w)=1\otimes w-\Delta(w)+w\otimes 1=u$. As a consequence, $\Delta(z'+w)=1\otimes (z'+w)+(z'+w)\otimes 1$, i.e. $z'+w$ is a primitive element in $H$. By the fact that $H_1=C_1$, $z'+w\in C_1$. But this would imply that $z'\in C$, which contradicts the choice of $z'$.

By Lemma \ref{cohomology}, the cohomology classes in $\h^2(\Omega C)$ represented by $u$ and $x\otimes y$ only differ by a non-zero scalar. Hence there exists $v\in C^+$ and $a\in k^\times$ such that $\partial^1(v)=au-x\otimes y$. Let $z=az'+v$. Then $z\notin C$ and $\Delta(z)=1\otimes z+x\otimes y+z\otimes 1$. 

Next, assume that $\GK H=3$. Now we have to show that $H$ is generated by $x$, $y$ and $z$. Let $K$ be the subalgebra of $H$ generated by $x$, $y$ and $z$. Then it is easy to check that $K$ is a sub-bialgebra and thus a Hopf subalgebra of $H$ by Lemma \ref{sub}. By the construction of $K$, $C\subsetneq K$. By Lemma \ref{GKplus1}, $\GK \gr K\ge \GK \gr C+1= 3$. On the other hand, $\GK\gr K=\GK K\le \GK H=3$ since $K\subset H$. Hence $\GK K=3$. Now it follows from Lemma \ref{GKcrit} that $K=H$. This completes the proof.
\end{proof}

\end{document}